\documentclass[reqno]{amsart}

\usepackage{IEEEtrantools}

\usepackage{graphicx}
\usepackage{amsmath} 
\usepackage{amsthm}
\usepackage{amssymb}

\usepackage{nicefrac}

\usepackage{tikz}
\usetikzlibrary{calc,shapes,arrows,positioning}

\usepackage{adjustbox}

\usepackage[
backend=bibtex,
maxnames=10,
bibencoding=utf8,
sorting=none,             
style=numeric-comp        
]{biblatex}
\defbibheading{bibliography}[\refname]{\section*{#1}}

\newtheorem{theorem}{Theorem}
\newtheorem{lemma}{Lemma}
\newtheorem{corollary}{Corollary}

\newtheorem{remark}{Remark}

\newcommand{\B}{\mathbb{B}}
\newcommand{\R}{\mathbb{R}}
\newcommand{\N}{\mathbb{N}}
\newcommand{\Z}{\mathbb{Z}}

\newcommand{\intcc}[1]{\ensuremath{{\left[#1\right]}}}

\newcommand{\intco}[1]{\ensuremath{{\left[#1\right[}}}
\newcommand{\intoo}[1]{\ensuremath{{\left]#1\right[}}}

\DeclareMathOperator{\pre}{pre}

\renewcommand{\emptyset}{{\varnothing}}

\date{}

\title[Computing Robust Controlled Invariant Sets of Linear Systems]{Computing Robust Controlled Invariant Sets\\ of Linear Systems}

\author{Matthias Rungger}

\author{Paulo Tabuada}


\begin{document}

\maketitle

\begin{abstract}
We consider controllable linear discrete-time systems with bounded perturbations and 
present two methods to compute robust controlled invariant sets. The first method tolerates an arbitrarily small constraint violation to
compute an arbitrarily precise outer approximation of the maximal robust
controlled invariant set, while the second method provides an inner
approximation.  The outer approximation scheme is
$\delta$-complete, given that the constraint sets are formulated as finite
unions of polytopes.
\end{abstract}

\section{Introduction}

Let us consider two matrices \mbox{$A\in \R^{n\times n}$},
\mbox{$B\in\R^{n\times m}$} with \mbox{$m\le n$} and a nonempty
set $ W\subseteq\R^n$.
Throughout this note, we analyze linear, time-invariant, discrete-time systems with
additive perturbations  described by the difference inclusion
\begin{IEEEeqnarray}{c'c}\label{e:sys}
  \xi(t+1)\in A\xi(t)+B\nu(t)+W, &\quad W\neq\emptyset
\end{IEEEeqnarray}
where \mbox{$\xi(t)\in \R^n$} and \mbox{$\nu(t)\in \R^m$} is
the \emph{state signal}, respectively, \emph{input signal} and $W$ is the set of 
\emph{disturbances}.  Here we slightly abuse notation and use $x+W$ instead of
$\{x\}+W$ to denote the Minkowski set
addition defined for two sets $P,Q\subseteq \R^n$ by $Q+P=\{y\in\R^n\mid
\exists_{q\in Q},\exists_{p\in P}\; y=q+p\}$.

In addition to
the dynamics, we consider
\emph{state constraints}  and
\emph{input constraints}  given by the compact sets
\begin{IEEEeqnarray}{c't'c}\label{e:constraints}
X\subseteq \R^n & and & U\subseteq \R^m.
\end{IEEEeqnarray}

We are interested in the computation of \emph{feedbacks}
that map states to admissible inputs 
\begin{IEEEeqnarray}{c}\label{e:fb}
  \mu:\R^n\rightrightarrows U 
\end{IEEEeqnarray}
which force the trajectories of~\eqref{e:sys} to evolve inside the state
constraint set $X$.  The double-arrow notation $\rightrightarrows$
indicates that $\mu$ is \emph{set-valued}, i.e., for $x\in \R^n$, the
image $\mu(x)$ is a subset of $U$, see~\cite[Ch.~5]{RockafellarWets09}.
Subsequently, we use $\mathcal{F}(U)$ to denote the set of all feedbacks that
satisfy for all $x\in X$: $u\in \mu(x)$ implies for all $x'\in Ax+Bu+W:\; \mu(x')\neq\emptyset$.

A {\it trajectory} of
\eqref{e:sys} and 
$\mu\in\mathcal{F}(U)$, with \emph{initial state} \mbox{$x\in \R^n$}, is a sequence
$\xi:\Z_{\ge0}\to \R^n$ that satisfies $\xi(0)=x$ and for which there exists $\nu:\Z_{\ge0}\to \R^m$ so that
$\nu(t)\in \mu(\xi(t))$ and~\eqref{e:sys} hold for all $t\in\Z_{\ge0}$.

It is well-known~\cite{Ber72} that the feedbacks of interest,   i.e., the maps $\mu$ that force every trajectory  of~\eqref{e:sys} and $\mu$ to evolve
inside $X$ for all time,  are
characterized by the maximal robust controlled invariant
set~\mbox{\cite{Blanchini99,BM08}},
also known as infinite reachable set~\cite{Ber72} or
discriminating kernel \mbox{\cite{Aub91,Cardaliaguet96}},
contained in $X$.

A set $R\subseteq \R^n$ is called \emph{robust controlled invariant}
w.r.t.~\eqref{e:sys} and~$U$, if there exists a feedback
  $\mu\in\mathcal{F}(U)$ so that for every trajectory $\xi$ of~\eqref{e:sys} and
  $\mu$ with initial state $\xi(0)\in R$ we have $\xi(t)\in R$ for all times
  $t\in\N$.
  We use $R(X)$ to denote the \emph{maximal robust controlled invariant set
  of~\eqref{e:sys} and~$U$} defined as the largest
  robust controlled invariant subset of $X$.\footnote{Note that
  the invariance property of a set is closed under union so that the maximal
  controlled invariant set is well-defined.}

Given $R(X)$, the map $C:\R^n\rightrightarrows U$ defined by
\begin{IEEEeqnarray}{c}\label{e:controller}
  C(x)=\{ u\in U\mid Ax+Bu +W\subseteq R(X) \}
\end{IEEEeqnarray}
characterizes all feedbacks of interest in the following sense:
Suppose that a feedback $\mu\in\mathcal{F}(U)$ enforces the
constraints $X$ on the system~\eqref{e:sys}, i.e., every trajectory $\xi$
of~\eqref{e:sys} and $\mu$ satisfies $\xi(t)\in X$ for all $t\in\Z_{\ge0}$, then we have 
for all $x\in \R^n$ the inclusion $\mu(x)\subseteq C(x)$, see e.g.~\cite[Thm.~1]{VSLS00}.
Therefore, it is sufficient to determine $R(X)$, whenever one is
interested in feedbacks that enforce the
constraints  $X$ and $U$ on~\eqref{e:sys}. 

Even though, set invariance 
has a rich history, see e.g.~\cite{GloverSchweppe71,Ber72,Cardaliaguet96,Blanchini94}, the computation of
$R(X)$ for most types of constraint sets $X$ and $U$, e.g. when $X$ and $U$ are given as a union of polytopes,  is still an open problem.
In this note, we propose two algorithms to compute an outer, respectively, inner
invariant approximation of the maximal robust controlled invariant set. Both
algorithms are obtained as modifications of the well-known dynamic programming approach to
the computation of the infinite reachable set~\cite{Ber72,Blanchini94,BM08}. Before we
provide a more detailed description of our contribution,
we review the state of the art on the computation of invariant sets of
linear systems.

Let $\pre(R)=\{x\in \R^n\mid \exists_{u\in U} Ax+Bu+W\subseteq R\}$
denote the set of states that are mapped into $R$ by the dynamics when the input
is appropriately chosen. In~\cite{Ber72}, Bertsekas introduced the iteration
\begin{IEEEeqnarray}{c'c}\label{e:iteration}
\begin{IEEEeqnarraybox}[][c]{c'c}
  R_0=X, &
  R_{i+1}=\pre(R_i)\cap X
\end{IEEEeqnarraybox}
\end{IEEEeqnarray}
and showed, 
that for every open set $\Omega$ that contains $R(X)$, there exists
$j$ so that $R_{j}\subseteq \Omega$ for all $j\ge i$, provided that the
sets $R_i$ are nonempty. In our case, this implies the convergence
\begin{IEEEeqnarray}{c}\label{e:conv}
  R(X)=\lim_{i\to\infty} R_i
\end{IEEEeqnarray}
with respect to the Hausdorff distance. See also~\cite{GloverSchweppe71}.

The set convergence~\eqref{e:conv} shows that the maximal robust controlled invariant set $R(X)$
can, in principle, be outer approximated by the sequence $(R_i)_{i\in\Z_{\ge0}}$ with arbitrary precision. Nevertheless, even if 
the sets $R_i$ are computable, the approximation is not very
useful since in general the sets $R_i$ are not robust
controlled invariant and it
is not possible to derive a feedback from any $R_i$ that ensures that the system
always evolves inside the state constraint set.

However, in some cases it is possible to determine the maximal robust
controlled invariant set by the
iteration~\eqref{e:iteration}. If there exists $i\in\Z_{\ge0}$ so that two consecutive
iterations in~\eqref{e:iteration} result in equal sets, i.e., $R_{i+1}=R_i$,
then $R_i=R(X)$.
In this case, we
say that $R(X)$ is \emph{finitely
determined}~\cite[Lem.~2.1]{KerriganMaciejowski00}. 
Depending on the dynamics $(A,B)$ and the shape of $X$, $U$ and  $W$ there exist
conditions which ensure that $R(X)$ is finitely
determined, see~\cite{VSLS00}. A large class of cases is covered by the
following conditions. Suppose that $(A,B)$ is \emph{controllable},  i.e., the
controllability matrix $[B,\,AB,\ldots ,A^{n-1}B]$ has full rank~\cite{AM07}, then 
without loss of generality, we may assume that the system is in Brunovsky normal
form, also known as Controller Form, see~\cite[Sec.~6.4.1]{AM07}. In this
representation, if $W=\{0\}$ and the sets $X$ and $U$ are given by a finite
union of {\it hyper-rectangles}, then 
the maximal control
invariant set is finitely determined, see~\cite{VSLS00,TP06,RMT13}. 

Unfortunately, for one of the most popular settings, where $(A,B)$ is assumed to
be controllable, $W=\{0\}$ and
the sets $X$ and $U$ are assumed to be polytopes
with the origin in the
interior, $R(X)$ is not finitely determined. Nevertheless, in this case, one can
modify the iteration~\eqref{e:iteration} and set $R_0=\{0\}$ (instead
of $R_0=X$). As a result, each set $R_i$ is controlled invariant and in fact
$R_i$ is the $i$-step null-controllable
set~\cite{GutmanCwikel84,DarupMonnigmann14} and the union of the sets $R_i$
converges to the \emph{largest null-controllable set} $N(X)$, i.e., the set of
all initial states from which the system can be forced to the origin in finite
time without violating the constrains.
As~$R_i$ converges to the maximal null controllable set $N(X)$ and the closure
of $N(X)$ equals $R(X)$, see~\cite[Prop.~1]{DarupMonnigmann14}, the
iteration~\eqref{e:iteration} with $R_0=\{0\}$ provides an algorithm for the 
arbitrarily precise (inner) approximation of $R(X)$, with the considerable advantage
that the approximation is robust controlled invariant. Moreover, this approach provides a
so-called \emph{anytime algorithm}, i.e., for each iteration $i\in\Z_{\ge0}$ the set $R_i$
is controlled invariant and a feedback can be derived, which enforces the
trajectories of~\eqref{e:sys} with initial state in $R_i$ to
evolve inside the constraint set $X$. Additionally, due to the convergence of $R_i$,
the mismatch between $R_i$ and $R(X)$ decreases as the computation continues.

An alternative modification of the iteration~\eqref{e:iteration},  which also
provides an invariant approximation of $R(X)$ and is not restricted to $W=\{0\}$, is presented
in~\cite{Blanchini94, BlanchiniMesquineMiani95} and~\cite[Sec.~5.2]{BM08}. 
The set iteration, with initial set $X$, is given by 
\begin{IEEEeqnarray}{c'c}\label{e:contraction}
  R_0=X,&R_{i+1}=\pre(\lambda R_i)\cap X
\end{IEEEeqnarray}
for some contraction factor
$\lambda\in\intoo{0,1}$, where  $\lambda P$
  for $\lambda\in\R_{\ge0}$ and $P\subseteq\R^n$ is defined by $ \lambda
P=\{x\in\R^n\mid \exists_{p\in P}\; x=\lambda p\}$.
The computation of $(R_i)_{i\in\Z_{\ge0}}$ terminates, once the inclusion
$R_i\subseteq \hat\lambda/\lambda R_{i+1}$ holds for $\hat \lambda\in\intoo{\lambda,1}$.
Given that $X$ contains a $\lambda$-contractive 
convex set (see \cite[Def.~4.18]{BM08}) with the origin
in its interior, it is shown in~\cite[Prop.~5.9]{BM08} that there exists
$i\in\Z_{\ge0}$ so that the termination condition is satisfied $R_i\subseteq
\hat\lambda/\lambda R_{i+1}$  and $R_i$ is robust controlled invariant, in fact $R_i$ is
$\hat\lambda$-contractive, see also~\cite[Thm.~3.2]{Blanchini94}.

In this note, we assume that the dynamics $(A,B)$ are controllable and the
constraint sets $X$ and $U$ are compact. Under these assumptions, we
 provide two novel results for the outer as well as inner
approximation of $R(X)$.
For the outer invariant approximation of $R(X)$, we use the set iteration~\eqref{e:iteration}
and modify the stopping criterion in
\cite[Eq.~(5.17)]{BM08} to
\begin{IEEEeqnarray}{c}\label{e:stop1}
R_i\subseteq R_{i+n}+\varepsilon\B,
\end{IEEEeqnarray}
 where $\B$ denotes the closed unit ball in $\R^n$ w.r.t. to the
infinity norm $|\cdot|$.
We show that for every $\varepsilon\in\R_{>0}$ there exists an $i\in\Z_{\ge0}$
so that~\eqref{e:stop1} holds. Based on the set $R_{i+n}$, we derive a
{\it $\delta$-relaxed robust control
invariant set $R$}, i.e., $R(X)\subseteq R\subseteq X+\delta\B$ and $R$ is robust
controlled invariant w.r.t.~\eqref{e:sys} and $U+\delta\B$. Here $\delta=c\varepsilon$, where $c\in\R_{\ge0}$ is a constant that is known a-priori and the relaxation of the
constraints can be made arbitrarily small by choosing an appropriate
\mbox{$\varepsilon\in\R_{>0}$}. Moreover, we show that the set $R$ converges to
$R(X)$ as $\varepsilon$ decreases to
zero.  Note that this approach can also be used in an anytime scheme. In that situation, at each
iteration $i\ge n$, we determine $\varepsilon\in\R_{\ge0}$ so
that~\eqref{e:stop1} holds. If the constraint relaxation $\delta$
is tolerable, we stop the computation, otherwise, we continue with
$R_{i+1}$.

For the inner invariant approximation of $R(X)$, we modify the 
iteration~\eqref{e:contraction} to
\begin{IEEEeqnarray}{c'c}\label{e:miter}
  R_0=X,& R_{i+1}=\pre_\rho(R_i)\cap X
\end{IEEEeqnarray}
where  the map $\pre_\rho$ is defined for $\rho\in\R_{\ge0}$ by
\begin{IEEEeqnarray}{C}\label{e:prerho}
  \pre_\rho(R)=\{x\in\R^n\mid \exists_{u\in U}Ax+Bu+W+\rho\B\subseteq R\}.
  \IEEEeqnarraynumspace
\end{IEEEeqnarray}
Given $\rho\in \R_{>0}$, we show  that there exists $i\in\Z_{\ge0}$ so that
$R_{i}\subseteq R_{i+1}+\rho\B$ holds and that $R_{i+1}$ is robust controlled
invariant.  Moreover, we provide conditions which ensure that
$R_{i+1}$ is nonempty.
Although, the modification from
$\pre(\lambda P)$ to $\pre_\rho(P)$ is
rather straightforward, it has substantial effects. Not only allows this
modification to extend the idea of $\lambda$-contractive sets~\cite{Blanchini94}
from convex sets to non-convex sets,  but it also removes the requirement that
$X$ contains a convex $\lambda$-contractive set that contains the origin in its
interior.

In summary, compared to existing
approaches, we do not assume that the state constraint set contains a $\lambda$-contractive convex
set with the origin in its
interior~\cite{Blanchini94,BlanchiniMesquineMiani95,BM08}, nor do we impose any restrictions on the shape of the constraint
sets~\cite{VSLS00,TP06,RMT13}, 
neither do we assume $W=\{0\}$~\mbox{\cite{GutmanCwikel84,DarupMonnigmann14}},
but simply consider compact constraint sets and general disturbance sets.
Specifically,  
we allow sets given by finite unions of polytopes, i.e., the sets 
\mbox{$X_i\subseteq\R^n$},  \mbox{$U_j\subseteq\R^m$},
\mbox{$W_k\subseteq\R^n$}
with $i\in\intcc{1;I}$, $j\in\intcc{1;J}$,  $k\in\intcc{1;K}$ and \mbox{$I,
J,K\in\N$} are polytopes and
\begin{IEEEeqnarray}{c'c'c}\label{e:poly:constraints}
  X=\bigcup_{i\in\intcc{1;I}}X_i, &
  U=\bigcup_{j\in\intcc{1;J}} U_j,  &
  W=\bigcup_{k\in\intcc{1;K}} W_k.
  \IEEEeqnarraynumspace
\end{IEEEeqnarray}
In this case, the sets $(R_i)_{i\in\Z_{\ge0}}$ are
computable~\cite[Sec.~III.B]{RakovicKerriganMayneLygeros06}
and the proposed scheme for the  
outer invariant approximation is
$\delta$-complete~\cite{GaoAvigadClarke12}: Let $\delta\in\R_{>0}$, $(A,B)$ be controllable and 
$X$, $U$, $W\neq\emptyset$ be defined in~\eqref{e:poly:constraints}, then
the proposed algorithm either returns an empty set $R_{i+n}$, in which case
the set $R(X)$ is empty, or we obtain a $\delta$-relaxed robust controlled invariant set $R$.

We would like to point out that constrains sets in the form of~\eqref{e:poly:constraints} arise in a variety of
different situations, see e.g.~\cite{DallalColomboDelVecchioLafortune13}, 
and are particularly important in the synthesis problems with respect to 
safe linear temporal logic specifications~\cite{RMT13}.

\section{Outer Invariant Approximation}
\label{s:oa}

We begin with a lemma which shows that the stopping
criterion~\eqref{e:stop1} is valid. 

\begin{lemma}\label{l:termination}
Consider the system~\eqref{e:sys} and the
compact constraint sets in~\eqref{e:constraints}. Let
$(R_i)_{i\in\Z_{\ge0}}$ be defined
according to~\eqref{e:iteration}. Then for any $\varepsilon\in\R_{>0}$ there exists $i\in\Z_{\ge0}$
so that~\eqref{e:stop1} holds. 
\end{lemma}
\begin{proof}
  Since $R_i=\emptyset$ implies $R_j=\emptyset$ for all $j\ge
  i$, the assertion trivially holds since~\eqref{e:stop1} results in $\emptyset\subseteq \emptyset$
  for $i\in
  \Z_{\ge0}$ with $R_i=\emptyset$ and subsequently we assume $R_i\neq\emptyset$ for
  all  $i\in \Z_{\ge0}$.
  From~\eqref{e:conv} follows that there exists
$i'\in\Z_{\ge0}$
so that for all $i\ge i'$ we have $R_i\subseteq R(X)+\varepsilon\B$ and we
obtain $R_i\subseteq R_j+\varepsilon\B$ for any $j\in\Z_{\ge0}$ which
shows~\eqref{e:stop1}.
\end{proof}

In the following, we make use of {\it $\delta$-constraint $i$-step
null-controllable sets}
$N^\delta_i\subseteq \R^n$, i.e., the set of initial states from which 
the unperturbed system $\xi(t+1)=A\xi(t)+B\nu(t)$ can be forced to the origin while satisfying
the input and state constraints $U=\delta\B$ and $X=\delta\B$. Let $\delta\in\R_{>0}$, then we define the sequence
of sets $(N^\delta_i)_{i\in\Z_{\ge0}}$ recursively by 
\begin{IEEEeqnarray}{c}\label{e:nullcontrollable}
\begin{IEEEeqnarraybox}[][c]{rCl}
N^\delta_0&=&\{0\},\\ 
N^\delta_{i+1}&=&\{x\in\R^n\mid \exists_{u\in \delta\B}\;
Ax+Bu\in N^\delta_i\}\cap \delta \B.
\end{IEEEeqnarraybox}
\end{IEEEeqnarray}
Note that for a fixed $\delta\in\R_{>0}$ it is straightforward to compute the
sets $(N^\delta_i)$ by polyhedral projection and
intersection~\cite{BM08}.
We use the following technical lemma about $\delta$-constraint $i$-step
null-controllable sets.

\begin{lemma}\label{l:bounds}
Consider the system \eqref{e:sys} with $W=\{0\}$. Let $N_n^\delta$ be defined
according to~\eqref{e:nullcontrollable}. Suppose that $(A,B)$ is
controllable, then 
\begin{IEEEeqnarray}{c'c't}\label{e:delta}
\exists_{c\in\R_{>0}}\;
\forall_{\varepsilon\in \R_{>0}}:&
\varepsilon \B \subseteq N^\delta_n & with $\delta=c\varepsilon$. 
\end{IEEEeqnarray}
\end{lemma}
\begin{proof}
We show that there exists $c\in \R_{>0}$ such that for
every $x\in\R^n$ there exists $\nu:\intco{0;n}\to \R^m$ so that the trajectory
of $\xi(t+1)=A\xi(t)+B\nu(t)$ with $\xi(0)=x$ satisfies $\xi(n)=0$, and for all
$t\in\intco{0;n}$ we have $|\xi(t)|\le c|x|$ and $|\nu(t)|\le c|x|$. This
implies the assertion of the lemma, since it is easy to see that $\xi(t)\in N^\delta_{n-t}$
with $\delta\ge c|x|$
holds for all $t\in\intcc{0;n}$.
The trajectory at time $n$ is given by 
$\xi(n)=A^nx +\mathcal{C} V$, where $\mathcal{C}$ is the 
controllability matrix
$[B,\,AB\ldots A^{n-1}B]$ and $V$ is a vector in $\R^{mn}$ with
$V=[\nu(n-1)^\top,\; \ldots,\nu(0)^\top]^\top$. Let $\mathcal{C}'\in \R^{n\times
n}$ denote a matrix containing $n$ linearly independent columns of
$\mathcal{C}$. Such a matrix always exists, since $(A,B)$ is controllable and hence
$\mathcal{C}$ hast full rank. Given $x\in\R^n$, we determine the input sequence
$V$ by setting the entries $V'$ of $V$ associated with $\mathcal{C}'$ to
$V'=-(\mathcal{C}')^{-1}A^nx$ and the remaining entries of $V$ to zero. It
follows that $\xi(n)=A^nx +\mathcal{C} V=0$. Moreover, $|V'|\le
c'|x|$ with $c'=|(\mathcal{C}')^{-1}A^n|$ holds and $|\nu(t)|\le c'|x|$ for all
$t\in\intco{0;n}$ follows. From
$\xi(t)=A^t+\sum_{s=0}^{t-1}A^{t-(s+1)}B\nu(s)$ follows that $|\xi(t)|\le
(|A^t|+\sum_{s=0}^{t-1}|A^{t-(s+1)}B|c')|x|$ holds and the assertion follows.
\end{proof}
\begin{corollary}
Let $z_j\in\R^n$, $j\in\intcc{1;2^n}$ denote the vertices of the unit cube $\B$.
A constant $c\in\R_{>0}$ that satisfies~\eqref{e:delta} is given by
$c=\max_{j\in\intcc{1;2^n}} c_j$
where $c_j$ is obtained by solving the linear program 
\begin{IEEEeqnarray}{c}\label{e:linprog}
\begin{IEEEeqnarraybox}[][c]{t'l'l}
&&\min_{c_j,u_0,\ldots,u_{n-1}} c_j\\
s.t.
&& A^nz_j+\sum_{k=0}^{n-1} A^{n-k-1}Bu_k= 0\\
& \forall_{i\in\intcc{0;n-1}}&|u_i|\le c_j\\
& \textstyle
  \forall_{i\in\intcc{1;n-1}}&\left|A^iz_j+\sum_{k=0}^{i-1} A^{i-k-1}Bu_k\right|\le c_j.
\end{IEEEeqnarraybox}
\end{IEEEeqnarray}
\end{corollary}
Note that $|x|$ denotes  the infinite norm of $x\in\R^n$ and the corollary
follows by the linearity of the trajectories of~\eqref{e:sys}. 

We proceed with the main result related to the outer invariant approximation. 
\begin{theorem}\label{t:outer}
Consider the system~\eqref{e:sys} and 
compact constraint sets~\eqref{e:constraints}. Let $(A,B)$ be controllable and consider
the sequences of sets
$(R_i)_{i\in\Z_{\ge0}}$ and $(N^\delta_i)_{i\in\Z_{\ge0}}$ given according
to~\eqref{e:iteration}, respectively~\eqref{e:nullcontrollable}, with 
$\varepsilon\in \R_{>0}$, $\delta=c\varepsilon$ and $c$
satisfying~\eqref{e:delta}. Let $i^*\in\Z_{\ge0}$ be the smallest index, so
that~\eqref{e:stop1} holds. The set
\begin{IEEEeqnarray}{c}\label{e:oapprox}
  \textstyle
  R:=\bigcup_{j\in\intcc{1;n}} R_{i^*+j}+N_j^\delta
\end{IEEEeqnarray}
is a subset of $X+\delta\B$ and 
is robust controlled invariant w.r.t.~\eqref{e:sys} and $U+\delta\B$.
\end{theorem}
\begin{proof}
  Consider the set $R$ defined in~\eqref{e:oapprox}. If
  $R=\emptyset$ the assertion trivially holds (since the empty set is robust
  controlled invariant) and subsequently we consider $R\neq\emptyset$.  Due to the choice of
$\delta=c\varepsilon$ with $c$ satisfying~\eqref{e:delta} we have
$\varepsilon\B\subseteq \cup_{j\in\intcc{1;n}} N^\delta_j \subseteq \delta \B$, which
  together with $R_{i^*}\subseteq X$ implies that $R\subseteq X+\delta \B$. Moreover,
\eqref{e:stop1} and~\eqref{e:delta} imply $R_{i^*}\subseteq R$. We show that for
every $x\in R$ there exists $u\in U+\delta\B$ so that $Ax+Bu+W\subseteq R$ which
shows that $R$ is robust controlled invariant~\cite[Prop.~1, ii)]{RakovicKerriganMayneLygeros06}. Let $x\in R$, then there
exists $j\in\intcc{1;n}$ so that $x\in R_{i^*+j}+N_j^\delta$. Let $x=x_r+x_n$ so
that $x_r\in R_{i^*+j}$ and $x_n\in N_j^\delta$. Then there exists $u_r\in U$
and $u_n\in \delta\B$ so that $Ax_r+Bu_r+W\subseteq R_{i^*+j-1}$ and
$Ax_n+Bu_n\in N_{j-1}^\delta$ and it follows that $Ax+Bu+W\subseteq
R_{i^*+j-1}+N_{j-1}$ where $u=u_r+u_n\in U+\delta\B$. If $j\ge 2$, it follows from the definition of
$R$ that $Ax+Bu+W\subseteq R$. If $j=1$, we use~\eqref{e:stop1}
and~\eqref{e:delta} to get \mbox{$Ax+Bu+W\subseteq R_{i^*}\subseteq
R_{i^*+n}+\varepsilon\B\subseteq R$}.
\end{proof}

By decreasing
the stopping parameter $\varepsilon\in\R_{>0}$ the set $R$ defined
in~\eqref{e:oapprox} converges to $R(X)$ w.r.t.~the Hausdorff distance
$d_H(P,Q):=\inf\{\eta\in\R_{\ge0}\mid Q\subseteq P+\eta\B\wedge P\subseteq Q+\eta\B\}$.

\begin{corollary}
Consider the hypothesis of Theorem~\ref{t:outer}. Let  $R_\varepsilon$ denote the
set $R$ defined in~\eqref{e:oapprox} for parameter $\varepsilon\in\R_{>0}$ and
let $R(X)$ be the maximal robust controlled invariant set of~\eqref{e:sys} and $U$.
For any sequence $(\varepsilon_j)_{j\ge0}$ in $\R_{>0}$ with limit $0$ we
  either have $R_{\varepsilon_j}=\emptyset$ for some $j$ so that
  $R(X)=\emptyset$ follows, or we have $\lim_{\varepsilon \to 0, \varepsilon>0}d_H(R(X),R_{\varepsilon})=0$.
\end{corollary}
\begin{proof}
Consider the sequence $(R_i)_{i\in\Z_{\ge0}}$ 
according to~\eqref{e:iteration}.
Let $i^*(\varepsilon)$ denote the
smallest $i^*\in\Z_{\ge0}$ such that~\eqref{e:stop1} holds for a fixed
$\varepsilon\in\R_{>0}$.
Consider a sequence $(\varepsilon_j)_{j\ge 0}$ in
  $\R_{>0}$ that converges to zero. If $R_{\varepsilon_j}=\emptyset$ for some
  $j\ge0$, it follows from~\eqref{e:oapprox} that
  $R_{i^*(\varepsilon_j)+j'}=\emptyset$ for all $j'\in\intcc{1;n}$, and
  $R(X)=\emptyset$ follows. Subsequently we consider
  $R_{\varepsilon_j}\neq\emptyset$ for all $j\in\Z_{\ge0}$.
  From $\delta_j = c \varepsilon_j$ and
$\varepsilon\B\subseteq \cup_{j\in\intcc{1;n}} N^\delta_j \subseteq \delta \B$
follows $R_{i^*(\varepsilon_j)} \subseteq R_{\varepsilon_j} \subseteq
R_{i^*(\varepsilon_j)}+c\varepsilon_j\B$ and it is sufficient to show that
$R_{i^*(\varepsilon_j)}$ converges to $R(X)$.
We use the fact that $\varepsilon_{j'}<\varepsilon_{j}$
implies $i^*(\varepsilon_{j'})\ge i^*(\varepsilon_{j})$ and distinguish two
cases: 1) if $i^*(\varepsilon_j)\to \infty$ as $j\to \infty$ we
  use~\eqref{e:conv}
  to conclude $\lim_{j\to \infty} d_H(R_{i^*(\varepsilon_j)},R(X))=0$; 2) otherwise we can assume that
$i^*(\varepsilon_j)\to i'$ for some $i'\in \Z_{\ge0}$. Hence, there exists $j'\in\Z_{\ge0}$ such that
$i^*(\varepsilon_j)=i$ for all $j\ge j'$ and by~\eqref{e:stop1} we have $R_i\subseteq  R_{i+n}+ \varepsilon_j\B$ for all $j\ge j'$,
which implies $R_i=R_{i+n}$ and we get $R_i=R(X)$.
\end{proof}

\begin{remark}
Consider the system~\eqref{e:sys} and  the
compact constraint sets~\eqref{e:constraints}. Let $(A,B)$ be controllable and
fix $\varepsilon\in\R_{>0}$. 
Suppose that we have an algorithm to iteratively compute $R_i$ and check the
inclusion~\eqref{e:stop1}, as it is the case e.g. for sets given
by~\eqref{e:poly:constraints} see~\cite{RakovicKerriganMayneLygeros06,Baotic09}. Then It follows from Lemma~\ref{l:termination} that there
exists $i\in\Z_{\ge0}$ so that~\eqref{e:stop1} holds. If $R_{i+n}=\emptyset$, then there does not
exist a feedback to enforce the constraints $X$ and $U$, in particular
$R(X)=\emptyset$. If $R_{i+n}\neq\emptyset$, due to the controllability of
$(A,B)$ we can solve the linear
program~\eqref{e:linprog} and compute the sets
$(N^\delta_{i+j})_{j\in\intcc{1;n}}$ with which we construct the set $R$ according
to~\eqref{e:oapprox}. Then it follows from Theorem~\ref{t:outer} that $R$ is
robust controlled invariant and a feedback to enforce the constraints $X+\delta\B$ and
$U+\delta\B$ is derived from the map
\begin{IEEEeqnarray*}{c}
K(x)=\{u\in U+\delta\B \mid Ax + Bu +W\subseteq R\}.
\end{IEEEeqnarray*}
Since $R(X)\subseteq R$ it is straightforward to see that the map defined
in~\eqref{e:controller} satisfies $C(x)\subseteq K(x)$ for all $x\in R$.
\end{remark}

For polyhedral disturbances  and constraints
sets~\eqref{e:poly:constraints}, the set iterates $R_i$ can be effectively
computed and the inclusion can be effectively tested,
see~\cite[Sec.~III.B]{RakovicKerriganMayneLygeros06} and~\cite{Baotic09}. In
the worst case, the computational complexity of these operations grows
exponentially with $i$, see~\cite{Baotic09}. Nevertheless, we present in Section V a nontrivial
example where the proposed algorithm can be executed until termination.

\section{Inner Invariant Approximation}
\label{s:ia}

For the inner approximation of $R(X)$ we fix $\rho\in\R_{> 0}$ and analyze the sequence 
\begin{IEEEeqnarray}{c}\label{e:inner}
\begin{IEEEeqnarraybox}[][c]{c'c}
  R^\rho_0=X,&
  R^\rho_{i+1}=\pre_\rho(R^\rho_i)\cap X
\end{IEEEeqnarraybox}
\end{IEEEeqnarray}
where $\pre_\rho$ is defined in~\eqref{e:prerho}. The stopping criterion, as
proposed in~(5.17) in~\cite{BM08}, is given by 
\begin{IEEEeqnarray}{c}\label{e:stop2}
  R^\rho_i\subseteq R_{i+1}^\rho+\rho\B.
\end{IEEEeqnarray}

\begin{theorem}
Consider the system~\eqref{e:sys} and compact
constraint sets~\eqref{e:constraints}. Let $(R^\rho_i)_{i\in\Z_{\ge0}}$ be defined in~\eqref{e:inner}.
For every $\rho\in\R_{>0}$ there exists an index $i\in\Z_{\ge0}$ such
that~\eqref{e:stop2} holds and $R^\rho_{i+1}$ is robust controlled
invariant w.r.t.~\eqref{e:sys} and $U$.
\end{theorem}
\begin{proof}
The proof of the existence of $i\in\Z_{\ge0}$ so that~\eqref{e:stop2} holds, follows by
the same arguments as the proof of Lemma~\ref{l:termination} and is omitted here. 

If $R^\rho_{i+1}=\emptyset$ the assertion trivially holds and subsequently we
consider $R^\rho_{i+1}\neq\emptyset$.
Let $x\in R^\rho_{i+1}=\pre_\rho(R^\rho_i)\cap X$. There exists $u\in U$
such that $Ax+Bu+W+\rho\B\subseteq R^\rho_i\subseteq R^\rho_{i+1}+\rho\B$ which
implies that $Ax+Bu+W\subseteq  R^\rho_{i+1}$ and it follows that $R^\rho_{i+1}$
  is robust controlled invariant~\cite[Prop.~1, ii)]{RakovicKerriganMayneLygeros06}.
\end{proof}

Let $\varepsilon\in\R_{>0}$, in the following theorem we consider the strengthened constraint sets 
\begin{IEEEeqnarray}{c}\label{e:stronger}
\begin{IEEEeqnarraybox}[][c]{l}
\bar X_\varepsilon=\{x\in \R^n\mid x+\varepsilon\B\subseteq X\}\\
\bar U_\varepsilon=\{u\in \R^m\mid u+\varepsilon\B\subseteq U\}
\end{IEEEeqnarraybox}
\end{IEEEeqnarray}
and show that there exists a parameter $\rho\in\R_{>0}$
so that any robust controlled invariant set $\bar R_\varepsilon\subseteq \bar X_\varepsilon$
w.r.t.~\eqref{e:sys} and $\bar U_\varepsilon$ is a subset of $R_{i+1}^\rho$.

\begin{theorem}
Consider the system~\eqref{e:sys}, $(A,B)$ being
controllable and compact
constraint sets~\eqref{e:constraints}. 
Let $(R^\rho_i)_{i\in\Z_{\ge0}}$ be defined in~\eqref{e:inner}.
Let $\varepsilon\in\R_{>0}$, and consider the sets $\bar X_\varepsilon$ and
$\bar U_\varepsilon$ in~\eqref{e:stronger}. There exists $\rho\in\R_{>0}$ so that
for any set $\bar R_\varepsilon\subseteq \bar X_\varepsilon$ that satisfies
\begin{IEEEeqnarray}{c}\label{e:inner:inv}
  x\in\bar R_\varepsilon \implies \exists_{u \in\bar U_\varepsilon}: Ax+Bu+W\subseteq \bar X_\varepsilon
\end{IEEEeqnarray}
we have $\bar R_\varepsilon\subseteq R_{i+1}^\rho$, where $i\in\Z_{\ge0}$ satisfies~\eqref{e:stop2}.
\end{theorem}
\begin{proof}
Let us consider the system 
\begin{IEEEeqnarray}{c}\label{e:modified:sys}
\xi(t+1)=A\xi(t)+B\nu(t)+W+\rho\B.
\end{IEEEeqnarray}
Let $R^\rho(X)$ be the maximal robust controlled invariant set of~\eqref{e:modified:sys} and $U$.
From the definition of $\pre_\rho$ in~\eqref{e:inner} we see that 
$R^\rho(X)\subseteq R_i^\rho$ for every $i\in\Z_{\ge0}$. 
  Let $\varepsilon\in\R_{>0}$. In the following, we consider $\bar
  R_{\varepsilon}\neq\emptyset$ (otherwise the assertion trivially holds) and 
show that there exists $\rho\in \R_{>0}$ and a set $K$ with $\bar
R_\varepsilon\subseteq K\subseteq R^\rho(X)$, which proves the theorem.

Let $\delta=\varepsilon/n$ and $\rho\in\R_{>0}$ so
that $c \rho =\delta $, where the constant $c$ is chosen according to
Lemma~\ref{l:bounds} (which is applicable, since $(A,B)$ is controllable).
Consider $N^\delta_i$, $i\in\intcc{0;n}$ defined according
to~\eqref{e:nullcontrollable}. Note that~\eqref{e:delta} implies that
$\rho\B\subseteq N^\delta_n$. We define the set
$K:=\bar R_\varepsilon+\sum_{i=1}^n N^\delta_i$. Note that
$N_i^\delta\subseteq\delta \B$ holds for every $i\in\intcc{1;n}$, which together
with $\bar R_\varepsilon+\varepsilon\B\subseteq X$ and $\delta=\varepsilon/n$,
implies  $K\subseteq X$. We show that $K$ is robust controlled invariant
w.r.t.~\eqref{e:modified:sys} and $U$. Let $x\in K$, and pick
$x_r\in\bar R_\varepsilon$ and
$x_i\in N_i^\delta$, $i\in\intcc{1;n}$ so that $x=x_r+\sum_{i=1}^n x_i$. Since
$\bar R_\varepsilon$ satisfies~\eqref{e:inner:inv}, we can
pick $u_r\in \bar U_\varepsilon$ so that
$Ax_r+Bu_r+W\subseteq \bar R_\varepsilon$, which
implies that  $Ax_r+Bu_r+W+\rho\B\subseteq \bar R_\varepsilon +N_n^\delta$.
Moreover, for $x_i\in N_i^\delta$, we pick $u_i\in \delta\B$
so that $Ax_i+Bu_i\in N_{i-1}^\delta$. Let $u=u_r+\sum_{i=1}^n u_i$. As $u_r\in
\bar U_\varepsilon$ and
$\delta\le \varepsilon/n$  we have $u\in U$. We see that
$Ax+Bu+W+\rho\B\subseteq K$ holds, which shows $K\subseteq R^\rho(X)$.
\end{proof}

\section{An illustrative example}

We proceed with a simple example taken from~\cite{Vidal00} to illustrate our
results. We consider the system~\eqref{e:sys} with parameters
\begin{IEEEeqnarray*}{c/c/c}
A=
\begin{bmatrix} 0 & 1 \\ 1 & 1 \end{bmatrix},
&
B=
\begin{bmatrix} 0 \\ 1 \end{bmatrix},
&
W=\left\{
\begin{bmatrix} 1 \\ 1 \end{bmatrix}\alpha\in\R^2\;\middle|\;
\alpha\in\intcc{-1,1}
\right\}.
\end{IEEEeqnarray*}
The constraint sets are given by 
$U=\intcc{-100,100}$ and  $X=\{ x\in \R^2\mid Hx\le h_0\}$ with
\begin{IEEEeqnarray*}{c'c}
H= 
\left[
\begin{IEEEeqnarraybox*}[][c]{,r/r,}
1 & 1 \\
-3 & 1 \\
0 & -1 %
\end{IEEEeqnarraybox*}
\right]
,&
h_0=
\left[
\begin{IEEEeqnarraybox*}[][c]{,r,}
100 \\ -50 \\ -26%
\end{IEEEeqnarraybox*}
\right].
\end{IEEEeqnarray*}
For this particular example we are able to analytically compute the set
iterations $(R_i)_{i\in \Z_{\ge0}}$ defined in~\eqref{e:iteration}.
Specifically, the sets $(R_i)_{i\in \Z_{\ge0}}$ and $W$ are polytopes and we 
follow the approach  in~\cite{Ber72} to compute $\pre(R_i)$ in
terms of the Pontryagin set difference $R_i\sim W=\{x\in R_i\mid x+W\subseteq
R_i\}$, i.e.,
\begin{IEEEeqnarray*}{c}
\pre(R_i)=\{x\in \R^2\mid \exists_{u\in U} Ax+Bu \in (R_i\sim W)\}.
\end{IEEEeqnarray*}
See also \cite[Sec.~3.3]{Kerrigan00}. For $R_0=X$, we apply~\cite[Thm.~2.4]{KG98}, and obtain the difference $R_0\sim W=\{
  x\in \R^2\mid Hx\le h_0'\}$ with 
$h_0'=
\left[
98,\; -52,\; -27%
\right]^\top$
and $\pre(R_0)$ follows simply by projecting the polytope 
\begin{IEEEeqnarray*}{c}
\left\{
  (x,u)\in \R^3\;\middle|\;
\left[
\begin{IEEEeqnarraybox*}[][c]{,r/r,}
HA & HB \\
0 & 1 \\
0 & -1 %
\end{IEEEeqnarraybox*}
\right]
\left[
\begin{IEEEeqnarraybox*}[][c]{,r,}
x \\
u %
\end{IEEEeqnarraybox*}
\right]
\le
\left[
\begin{IEEEeqnarraybox*}[][c]{,r,}
h_0' \\
100 \\
100 %
\end{IEEEeqnarraybox*}
\right]
\right\}
\end{IEEEeqnarray*}
onto its first two coordinates. After the intersection of $\pre(R_0)$ with $R_0$
we obtain $R_1=\{x\in\R^2\mid Hx\le h_1\}$ with 
\begin{IEEEeqnarray*}{c}
h_1=
\left[
100,\; -50,\; -26-\tfrac{1}{3}%
\right]^\top.
\end{IEEEeqnarray*}
We repeat this computation and obtain the sequence of sets by
$R_i=\{x\in\R^2\mid Hx\le h_i\}$ with
\begin{IEEEeqnarray*}{c}
  \textstyle
h_i=
\left[
100,\; -50,\; -25-\sum_{j=0}^i\tfrac{1}{3^i}%
 \right]^\top
\end{IEEEeqnarray*}
whose limit is given by  $R(X)=\{x\in\R^2\mid Hx\le h\}$ with
\begin{IEEEeqnarray*}{c}
h=
\left[
100,\; -50,\; -26.5%
 \right]^\top.
\end{IEEEeqnarray*}
The boundary of the maximal robust controlled invariant set
$R(X)$ is illustrated in Figure~\ref{f:outer} and~\ref{f:inner} by the dotted
line.

Note that $R(X)$ is not finitely
determined, $X$ does contain the origin in its interior, nor is $W=\{0\}$. Hence, it is not
possible to apply any of the methods in~\cite{Vidal00,TP06,GutmanCwikel84, Blanchini94}, 
 to invariantly
approximate the maximal robust controlled invariant set. In the following we
apply the results from Sections~\ref{s:oa} and~\ref{s:ia} to compute
outer and inner invariant approximations of $R(X)$.

{\bf Outer approximation.}
We start by solving the linear program~\eqref{e:linprog}
to determine the constant $c=2$ which satisfies~\eqref{e:delta}. The
$\delta$-constraint $i$-step null controllable sets $N_j^\delta$ for
$j\in\intcc{1;2}$ are illustrated in Figure~\ref{f:f2}.
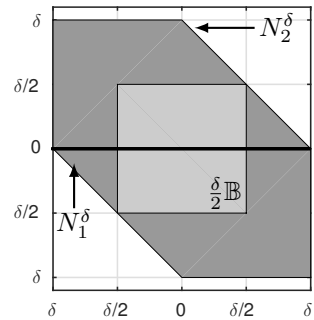
\begin{figure}[h]
\adjustbox{valign=t}{
\begin{minipage}{0.45\columnwidth}
\caption{The
$\delta$-constraint $1$-step (thick black bar) and $2$-step (dark gray polytope)
null controllable sets $N_j^\delta$ containing the
ball $\tfrac{\delta}{2}\B$ (light gray box).}\label{f:f2}
\end{minipage}}%
\adjustbox{valign=t}{
\begin{minipage}{0.5\columnwidth}
\centering
\begin{tikzpicture}[>=latex]
\node at (0,0) {\includegraphics[width=4cm]{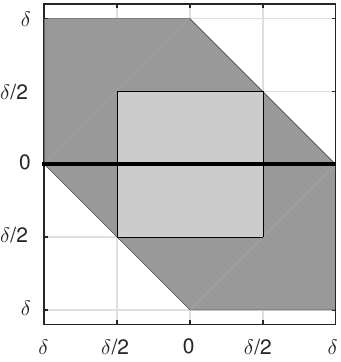}};
\node at (.8,-.4) {$\tfrac{\delta}{2}\B$};
\node at (-1.2,-.8) {$N^\delta_1$};
\draw[thick,->] (-1.2,-.6) -- (-1.2,0);
\node at (1.5,1.8) {$N^\delta_2$};
\draw[thick,->] (1.2,1.8) -- (0.4,1.8);
\end{tikzpicture}
\end{minipage}
}
\end{figure}
From the previous consideration it is straightforward to see that 
\begin{IEEEeqnarraybox}{c}
R_{i}\subseteq R_{i+2}+\tfrac{4}{3^{i+2}}\B
\end{IEEEeqnarraybox}
holds for all $i\in\Z_{\ge 0}$. Hence, in each iteration the stopping parameter is
given by $\varepsilon=4/3^{i+2}$. We illustrate the robust controlled invariant set defined
in~\eqref{e:oapprox} for $i=0$ and $i=3$ relative to $R(X)$ in
Figure~\ref{f:outer}. For $i=3$, $\delta=8/243$ and $R$ in Figure~\ref{f:outer} is
indistinguishable form $R(X)$.

\begin{figure}[h]
\centering
\includegraphics[width=8cm]{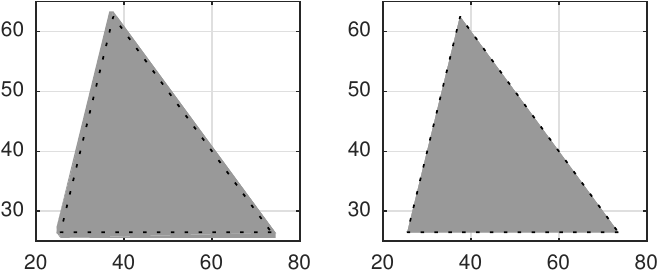}
\caption{Invariant outer approximations of $R(X)$ given according to~\eqref{e:oapprox} for
$i=0$ (left) and $i=3$ (right). The dotted line indicates $R(X)$.}\label{f:outer}
\end{figure}

{\bf Inner approximation.}
In order to obtain an inner approximation of $R(X)$, we compute the sequence of sets 
$(R_i^\rho)_{i\in\Z_{\ge0}}$ defined in~\eqref{e:inner}. Similar as before, we
compute $\pre_\rho(R^\rho_i)$ by using the Pontryagin set difference, i.e.,
\begin{IEEEeqnarray*}{c}
\pre_\rho(R^\rho_i)=\{x\in \R^2\mid \exists_{u\in U} Ax+Bu \in (R_i\sim (W+\rho\B))\}.
\end{IEEEeqnarray*}
We apply again~\cite[Thm.~2.4]{KG98} to compute $R_i\sim (W+\rho\B)$. Two
invariant inner approximations of $R(X)$ with parameters $\rho=1$ and
$\rho=1/10$ are illustrated in Figure~\ref{f:inner}. 

\begin{figure}[h]
\centering
\includegraphics[width=8cm]{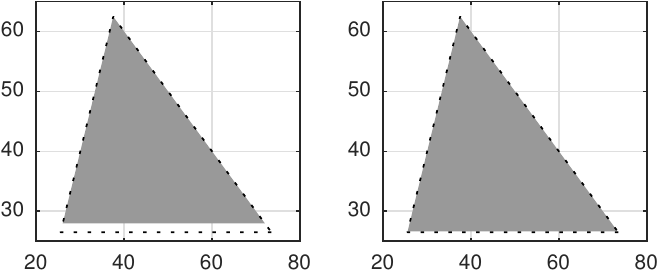}
\caption{Invariant inner approximations of $R(X)$ with parameters $\rho=1$
(left) and
$\rho=1/10$ (right). The dotted line indicates $R(X)$.}\label{f:inner}
\end{figure}

\section{Numerical Experiments}
\label{s:app}

We continue with the approximation of the maximal robust controlled invariant
set for a more complex system. To this end, we consider the linear dynamics used
in~\cite{KuwataSchouwenaarRichardsHow05} to model a rotor craft.
The system consists of four states. The first two states represent the position and the last two
states represent the velocity of the rotor craft. The
acceleration is considered as the input of the system.
The parameters of the differential inclusion are given by
\begin{IEEEeqnarray*}{l}
A=
\textstyle
\begin{bmatrix} I_2 & \tau I_2 \\ 0 & I_2 \end{bmatrix},
\;
B=
\begin{bmatrix} \tfrac{\tau^2}{2} I_2 \\ \tau I_2 \end{bmatrix},\\
\textstyle
  W=(\intcc{-\nicefrac{\tau^2}{2},\nicefrac{\tau^2}{2}}^2\times
  \intcc{-\tau,\tau}^2)w_{\rm max}
\end{IEEEeqnarray*}
where $I_2$ denotes the 2-dimensional identity matrix and $\tau=2.6\;$sec.
The state constraint set is given by
\begin{IEEEeqnarray*}{c't'c}
  X=(\intcc{-35,5}\times \intcc{-10,10}\times \intcc{-v_{\rm max},v_{\rm
  max}}^2)\smallsetminus O
\end{IEEEeqnarray*}
with $O=\cup_{i=1}^p o_i
+\intcc{-4,4}\times\intcc{-1,1}\times\intcc{-v_{\rm max},v_{\rm max}}^2$
representing some obstacles. The
first two coordinates of the centers of the obstacles are randomly generated integer values 
\begin{IEEEeqnarray*}{c}
\left[
\begin{IEEEeqnarraybox*}[][c]{r/r/r/r/r/r/r/r/r/r}
 -8&  -5&  -13&  -22&  -11&  -22&   -10& -2&  -15&  -17\\
  1&   5&  -8 &   -2 &   4 &   3 &   -1 &  7&  -6 & 5
\end{IEEEeqnarraybox*}
\right]
\end{IEEEeqnarray*}
while the last two coordinates of $o_i$ are set to zero.
The
input is constrained to $U=\intcc{-a_{\rm max},a_{\rm max}}$. 
We follow~\cite{KuwataSchouwenaarRichardsHow05} and set 
\begin{IEEEeqnarray*}{c/c/t/c}
  v_{\rm max}=0.5,&a_{\rm max}=0.17&and &w_{\rm max}\in \bar w\cdot a_{\rm max}
\end{IEEEeqnarray*}
where the perturbation level ranges over $\bar w\in \{0,.1,.2\}$. 
Using the control input
$u=-\intcc{1/\tau^2I_2\;3/2\tau^2I_2}^\top$, all states in the unit cube of the unperturbed system can be steered to the origin
in two steps without leaving the unit cube. Hence, a constant which
satisfies~\eqref{e:delta} is given by $c=1$. 
Moreover, in the subsequent computation of the outer approximation of the maximal controlled invariant set we can use $R_i\subseteq
R_{i+2}+\varepsilon\B$ as stopping criterion.

In the conducted experiments, in addition to the different perturbation levels
$\bar w\in\{0,.1,.2\}$,  we vary  the number of obstacles $p\in\{0,5,10\}$. The
approximation accuracy is set to $\varepsilon=0.01$.
For each computation, we display in Figure~\ref{f:ex2} the run-times of the
computation to determine the set $R_i$ as well as the numbers of halfspaces
$\#R_i$ used to represent the set $R_i$. The number of iterations until
termination can be deduced by the last shown data-point. For example, for $\bar
w=.1$ and $p=5$ ({\tikz[baseline=-0.5ex]{\node {\protect\includegraphics[width=.5cm]{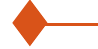}};}}),
we see only three data-points in the upper, middle subplot, which indicates that
at time $i=3$ the termination criterion $R_3\subseteq R_1+0.01\B$ holds.
Although, the worst case estimates predict that the number of halfspaces
necessary to represent the sets $R_i$ increases exponentially with the number of
iterations, see e.g.~\cite{Baotic09}, we do not observe such an increase in our
experiments and are therefore able to successfully approximate $R(X)$ for this example.

\begin{figure}[t]
  \begin{tikzpicture}
    \node at (.75,0) {\includegraphics[width=8cm]{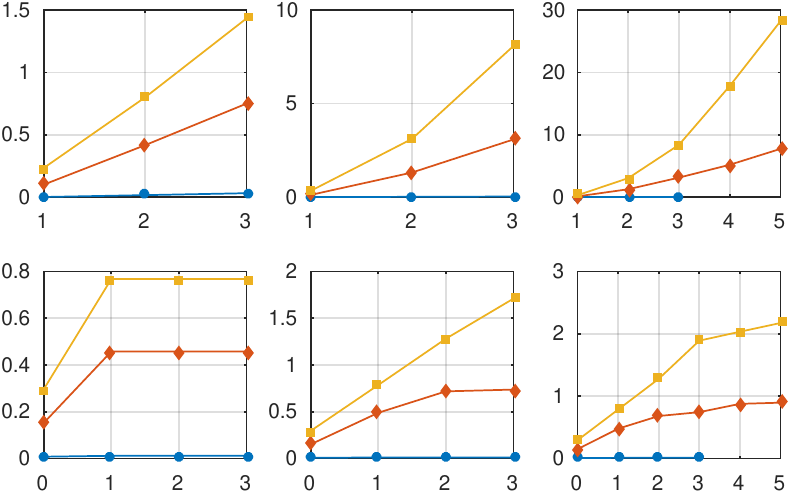}};
    \node[rotate=90] at (-3.6,1.5) {\footnotesize $t$[min]};
    \node[rotate=90] at (-3.6,-1) {\footnotesize $\#R_i/10^3$};
    \node at (-1.75,2.725) {\footnotesize $\bar w=0$};
    \node at  ( 1,2.725) {\footnotesize $\bar w=0.1$};
    \node at  (3.75,2.725) {\footnotesize $\bar w=0.2$};
  \end{tikzpicture}%
  \caption{Computation numbers for varying perturbation levels $\bar w\in\{0,.1,.2\}$ and number of obstacles
    $p=0$ ({\protect \tikz[baseline=-0.5ex]{\protect\node {\protect\includegraphics[width=.5cm]{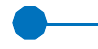}};}}),
    $p=5$ ({\protect \tikz[baseline=-0.5ex]{\protect\node {\protect\includegraphics[width=.5cm]{legend2}};}}),
    $p=10$ ({\protect \tikz[baseline=-0.5ex]{\protect\node {\protect\includegraphics[width=.5cm]{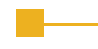}};}}).
    The upper row shows the run-times for the computation of the sets $R_i$. The lower row shows the number of half-spaces used to represent the sets $R_i$.}
    \label{f:ex2}
\end{figure}

All the computations were conducted on a single core of an Intel i7 $3.5$GHz CPU with $32$GB
memory, using MATLAB and the freely available {\tt Multi-Parametric Toolbox} \url{http://people.ee.ethz.ch/~mpt/2/}, which
provides all the polyhedral operations, necessary to compute the set iterates
$R_i$ and to check the set inclusion~\eqref{e:stop1}.

\newpage
\printbibliography
\end{document}